\documentclass[leqno,8pt]{article}
\usepackage[includehead,includefoot,margin=16mm]{geometry}
\usepackage{amsmath}
\usepackage{amsthm}
\usepackage{amsfonts}  
\usepackage{amssymb}
\usepackage{amsmath}
\usepackage{mathrsfs}
\usepackage{tikz}
\usepackage[english,francais]{babel}
\usepackage{hyperref}

\usetikzlibrary{matrix,arrows,shadings}
\pdfpageattr{/Group <</S /Transparency /I true /CS /DeviceRGB>>} 

\title{Singularit\'es canoniques et actions horosph\'eriques}
\author{Kevin Langlois\footnote{{\em Remerciement} : Le pr\'esent texte est financ\'e par l'universit\'e Heinrich Heine de Duesseldorf.  {\em Adresse} : Mathematisches Institut, Heinrich Heine Universit\"at, 40225 D\"usseldorf, Allemagne. {\em Courrier \'electronique} : langlois.kevin18@gmail.com.}}
\date{}

\begin{document}
\maketitle

\theoremstyle{plain}
\newtheorem{theorem}{Th\'eor\`eme}[section]
\newtheorem{lemme}[theorem]{Lemme}
\newtheorem{proposition}[theorem]{Proposition}
\newtheorem{corollaire}[theorem]{Corollaire}
\newtheorem*{theorem*}{Théorème}

\theoremstyle{definition}
\newtheorem{definition}[theorem]{D\'efinition}
\newtheorem{rappel}[theorem]{}
\newtheorem{conjecture}[theorem]{Conjecture}
\newtheorem{exemple}[theorem]{Exemple}
\newtheorem{notation}[theorem]{Notation}

\theoremstyle{remark}
\newtheorem{remarque}[theorem]{Remarque}
\newtheorem{note}[theorem]{Note}

\def\QQ{{\mathbb Q}}
\def\G{{\mathbb G}}
\def\ZZ{{\mathbb Z}}
\def\PP{{\mathbb P}}
\def\CC{{\mathbb C}}
\def\F{{\mathscr F}}
\def\V{{\mathscr V}}
\def\D{{\mathscr D}}
\def\Hom{{\rm Hom}}
\def\E{{\mathcal{E}_{\rm st}}}
\def\L{{\mathscr{L}}}
\def\LL{{\mathbb{L}}}
\def\ord{{\mathrm{ord}}}
\def\supp{{\mathrm{supp}}}
\def\M{{\mathcal{M}}}
\def\deg{{\mathrm{deg}}}
\def\supp{{\mathrm{supp}}}
\def\N{{\mathscr{E}}}
\def\vv{{\underline{v}}}
\def\div{{\mathrm{div}}}
\def\ord{{\mathrm{ord}}}
\selectlanguage{english}
\begin{abstract}
Let $G$ be a connected reductive linear algebraic group. We consider the normal $G$-varieties with horospherical orbits. In this short note, we provide a criterion to determine whether these varieties have at most canonical, log canonical or terminal singularities in the case where they admit an algebraic curve as rational quotient.
This result seems to be new in the special setting of torus actions with general orbits of codimension $1$. For the given $G$-variety $X$, our criterion is expressed in terms of a weight function $\omega_{X}$ that is constructed from the set of $G$-invariant valuations of the function field $k(X)$. In the log terminal case, the generating function of $\omega_{X}$ coincides with the stringy motivic volume of $X$. As an application, we discuss the case of normal $k^{\star}$-surfaces.
\end{abstract}

\selectlanguage{francais}
\begin{abstract}
Soit $G$ un groupe alg\'ebrique lin\'eaire r\'eductif connexe.
Nous consid\'erons les $G$-vari\'et\'es normales avec orbites horosph\'eriques. Dans cette courte note, nous donnons
 un crit\`ere pour d\'eterminer lorsque ces vari\'et\'es ont au plus des singularit\'es canoniques, log canoniques ou terminales dans le cas o\`u elles admettent une courbe alg\'ebrique comme quotient rationnel. Ce r\'esutat semble nouveau pour le cas sp\'ecial des actions de tores alg\'ebriques avec orbites g\'en\'erales de codimension $1$. Pour la $G$-vari\'et\'e consid\'er\'ee $X$, notre crit\`ere est exprim\'e en terme d'une fonction de poids $\omega_{X}$ qui est construite \`a partir de l'ensemble des valuations $G$-invariantes du corps des fonctions $k(X)$. Dans le cas log terminal, la fonction g\'en\'eratrice de $\omega_{X}$ correspond au volume motivique des cordes de $X$. Comme application, nous traitons le cas des $k^{\star}$-surfaces normales.
\end{abstract}

\section*{\centerline{Introduction}}
Dans cet article, les vari\'et\'es alg\'ebriques et les groupes alg\'ebriques sont d\'efinis sur un corps alg\'ebriquement clos 
$k$ de caract\'eristique $0$. Un espace homog\`ene $G/H$ sous un groupe alg\'ebrique lin\'eaire r\'eductif connexe $G$ est dit \emph{horosph\'erique} si le sous-groupe ferm\'e $H\subseteq G$ contient un sous-groupe unipotent maximal. Dans ce cas, 
le normalisateur $P = N_{G}(H)$ de $H$ dans $G$ est un sous-groupe parabolique, $P/H$ est un tore alg\'ebrique et l'espace homog\`ene $G/H$ est donc r\'ealis\'e comme $P/H$-torseur au dessus de la vari\'et\'e des drapeaux $G/P$ (voir \cite[Section 2]{Pas08}). Rappelons qu'une $G$-vari\'et\'e normale est dite \emph{sph\'erique} si elle poss\`ede une orbite ouverte sous l'action d'un sous-groupe de Borel. Comme cons\'equence de la d\'ecomposition de Bruhat, l'espace homog\`ene horosph\'erique $G/H$ est sph\'erique. 

Nous consid\'erons les $G$-vari\'et\'es normales avec orbites horosph\'eriques. Notre but est d'\'etablir un crit\`ere pour d\'eterminer lorsque que ces vari\'et\'es suppos\'ees $\QQ$-Gorenstein ont au plus des singularit\'es canoniques (log canoniques ou terminales) sous la condition d'existence d'un quotient rationnel qui est une courbe alg\'ebrique (cas de complexit\'e un).
Notre r\'esultat est motiv\'e par le cas des vari\'et\'es toriques (voir \cite{KKMS73, Dan78, Oda78}). Dans \cite{Rei80, Rei83}, les singularit\'es toriques intervenant dans le programme du mod\`ele minimal sont caract\'eris\'ees en termes de g\'eom\'etrie convexe, voir aussi \cite[Section 1]{BG95}, \cite[Sections 3, 4]{Dai02} pour une exposition.

Ces crit\`eres ont \'et\'e g\'en\'eralis\'es par Brion dans le cadre sph\'erique \cite{Bri93} en utilisant la th\'eorie de Luna-Vust (cf. \cite{Kno91}). Recemment, Liendo et Suess \cite{LS13} ont \'etudi\'es les singularit\'es des vari\'et\'es normales avec actions de tores alg\'ebriques via la description d'Altmann-Hausen (voir \cite{AH06, AHS08}). En particulier, ils obtiennent un crit\`ere pour les singularit\'es log terminales des vari\'et\'es normales $\QQ$-Gorenstein dot\'ees d'une op\'eration d'un tore alg\'ebrique avec orbites g\'en\'erales de codimension $1$, voir \cite[Corollary 5.4]{LS13} et \cite[Theorem 4.9]{LS13}, \cite[Theorem 2.22]{LT16} pour des g\'en\'eralisations de ce crit\`ere.  
\\

{ \bf Notations}. Nous r\'eunissons maintenant les notations n\'ecessaires pour \'enoncer notre r\'esultat. Nous utiliserons l'approche de Timashev pour d\'ecrire les op\'erations de groupes alg\'ebriques r\'eductifs de complexit\'e un, voir \cite{Tim97}.
Nous fixons d\'esormais, une $G$-vari\'et\'e normale $X$ avec orbites horosph\'eriques ayant un quotient rationnel  $\pi: X\dashrightarrow C$ sous l'action de $G$, o\`u $C$ est une courbe alg\'ebrique compl\`ete lisse. En particulier, $C$ s'identifie avec l'ensemble des places du corps de fonctions $k(X)^{G}$. Pour un choix fix\'e d'un sous-groupe de Borel $B\subseteq G$, nous appellerons \emph{couleur} de $X$ un diviseur premier $B$-stable de $X$ qui n'est pas $G$-stable. Les couleurs de $X$ constituent un ensemble fini $\F$. D\'esignons par $\F_{X}\subseteq \F$ le sous-ensemble des couleurs contenant une $G$-orbite. Puisque notre probl\`eme est locale, on peut supposer qu'il existe un ouvert affine dense $B$-stable $X_{0}\subseteq X$ intersectant toute $G$-orbite de $X$ \cite[Theorem 1.3]{Kno91}. Sous cette condition, $X_{0}$ est le compl\'ementaire de la r\'eunion des couleurs appartenant \`a $\F\setminus \F_{X}$ \cite[Lemma 2.1]{LT16}. Par ailleurs, $X$ admet un mod\`ele birationnel \'equivariant de la forme $C\times G/H$, o\`u $G/H$ est un espace homog\`ene horosph\'erique (voir par exemple \cite[Satz 2.2]{Kno90}). Si $G/H\rightarrow G/P$ est la projection naturelle vers la vari\'et\'e des drapeaux, alors $\F$ est naturellement en bijection avec l'ensemble des diviseurs de Schubert de $G/P$. En particulier, $\F$ s'identifie \`a l'ensemble des racines simples $\Phi$ de $(G, B)$
qui ne proviennent pas de $P$.  Pour $\alpha\in \Phi$ le symbole $\alpha^{\vee}$ repr\'esentera la coracine associ\'ee et $a_{\alpha}$ le nombre entier $\langle \sum_{\beta\in\Phi}\beta, \alpha^{\vee}\rangle$. Nous d\'esignerons par $\Phi_{X}$ le sous-ensemble de $\Phi$ correspondant aux couleurs de $\F_{X}$.

En consid\'erant le radical unipotent $U$ de $B$, l'alg\`ebre des invariants $A:= k[X_{0}]^{U}$ est gradu\'ee par les poids provenant des fonctions rationnelles propres de $X$ sous l'op\'eration de $B$. Soient $M$ le r\'eseau des $B$-poids de $A$ (ou indiff\'eremment de $k(X)$), $N = \Hom(M, \ZZ)$ le dual et $N_{\QQ} = \QQ\otimes_{\ZZ}N$, $M_{\QQ} = \QQ\otimes_{\ZZ}M$ les $\QQ$-espaces vectoriels associ\'es. Le c\^one des poids $\sigma^{\vee}\subseteq M_{\QQ}$ de $A$ est le dual d'un c\^one poly\'edral saillant $\sigma\subseteq N_{\QQ}$ (i.e. ayant $0$ comme sommet). De plus, en choisissant un ensemble de fonctions rationnelles propres $\V =\{\chi^{m}\in k(X)^{\star}, m\in M\}$ sous $B$ avec les conditions $\chi^{m}\cdot \chi^{m'} = \chi^{m+m'}$ pour tous $m,m'\in M$, l'alg\`ebre $A$ est d\'ecrite par une unique fonction lin\'eaire par mor\c ceaux  
$$\D_{X, \V}:\sigma^{\vee}\rightarrow {\rm CaDiv}(C_{X}),\,m\mapsto \sum_{y\in C_{X}}\min_{v\in\D_{y}}\langle m , v\rangle \cdot [y]$$ $$\text{ d\'efinie par l'\'egalit\'e }A= A[C_{X},\D_{X,\V}] := \bigoplus_{m\in\sigma^{\vee}\cap M}H^{0}(C_{X}, \mathcal{O}_{C_{X}}(\D_{X,\V}(m))\cdot \chi^{m}.$$  
Ici $C_{X}\subseteq C$ est un ouvert dense, les sous-ensembles $\D_{y}\subseteq N_{\QQ}$ sont des poly\`edres avec c\^one de r\'ecession $\sigma$ \'etant pour presque tout $y\in C_{X}$ \'egaux \`a $\sigma$ et ${\rm CaDiv}(C_{X})$ est l'espace vectoriel des $\QQ$-diviseurs de Cartier sur $C_{X}$. Le couple  
$(\D_{X, \V}, \F_{X})$ est la contrepartie combinatoire de $X$ et sera appel\'e un \emph{diviseur poly\'edral colori\'e} associ\'e \`a $X$ (voir \cite[Section 2]{AH06} et \cite[Section 1.3]{LT16}). Nous renvoyons \`a \cite{Tim97} pour
la construction inverse d'une telle $G$-vari\'et\'e \`a partir d'un diviseur poly\'edral colori\'e abstrait. Le formalisme des diviseurs poly\'edraux a \'et\'e introduit dans \cite{AH06}, voir \cite{AHS08, Lan16}  pour des g\'en\'eralisations. Dans la suite, nous poserons $\supp(\D_{X,\V}) = \{y\in C_{X}\,|\, \D_{y}\neq \sigma\}$ et noterons par  $\deg(\D_{X,\V})$ le sous-ensemble de $N_{\QQ}$ qui est vide si $C_{X}\neq C$ et qui est \'egal au poly\`edre $\sum_{y\in C}\D_{y}$ sinon. Notons que l'on a toujours $\deg\,\D\subseteq \sigma$ (voir \cite[Example 2.12]{AH06}).
\\

{ \bf Int\'egration motivique et fonctions de poids}. Nous introduisons la fonction de poids $\omega_{X}$ attach\'ee \`a la donn\'ee $(\D_{X, \V}, \F_{X})$. Dans le cas o\`u les singularit\'es de $X$ sont au plus log terminales, cette fonction appara\^it dans le calcul du \emph{volume motivique des cordes} defini par l'int\'egrale motivique
$$ \E(X) = \int_{\L(X')}\LL^{-\ord K_{X'/X}}d\mu_{X'}.$$
Le symbole $K_{X'/X}$ d\'esigne un diviseur canonique relatif provenant d'une log r\'esolution et $\L(X')$ est le sch\'ema des arcs de la d\'esingularisation. L'int\'egrale $\E(X)$ appartient \`a une certaine modification et compl\'etion $\M$ du localis\'e de l'anneau de Grothendieck $K_{0}({\rm Var}_{k})[\LL^{-1}]$ de la cat\'egorie des $k$-vari\'et\'es par rapport \`a la classe de la droite affine $\LL := [\mathbb{A}^{1}_{k}]$ et ne depend ni du choix de la log r\'esolution ni de celui du diviseur canonique relatif $K_{X'/X}$. Nous renvoyons \`a \cite{DL99, Loe09, Vey06} pour plus de d\'etails sur l'int\'egration motivique et \`a \cite{Bat98} pour la construction des invariants des cordes. Plus precis\'ement, le r\'esultat principal de \cite{LPR16} est la formule suivante inspir\'ee par celle de Batyrev et Moreau dans \cite[Theorem 4.3]{BM13}:
$$\E(X) = [G/H]\left(\sum_{[y,\nu, \ell]\in |\D_{X, \V}|}  \zeta_{\ell}\cdot\LL^{\omega_{X}(y,\nu,\ell)} \right),$$
o\`u ici $\zeta_{\ell}$ est \'egale \`a la classe $[C_{X}\setminus \supp(\D_{X,\V})]$ si $\ell = 0$
et \`a $\LL -1$ si $\ell\geq 1$ vue dans le compl\'et\'e $\M$; le second membre \'etant interpr\'et\'e comme la \emph{fonction g\'en\'eratrice} de $\omega_{X}$.

Supposons que la $G$-vari\'et\'e $X$ avec orbites horosph\'eriques associ\'ee \`a la donn\'ee $(\D_{X, \V}, \F_{X})$ est $\QQ$-Gorenstein (et possiblement non log terminale). Nous renvoyons \`a \cite[Corollary 2.19]{LT16} pour un crit\`ere explicite de la condition $\QQ$-Gorenstein. On d\'efinit alors l'ensemble $|\D_{X,\V}|$ par l'\'egalit\'e
 $$|\D_{X,\V}| = \{[y,\nu, \ell]\,|\, y\in \supp(\D_{X,\V}), \nu\in N, \ell\in \ZZ_{\geq 0}\text{ et }(\nu, \ell )\in C(\D_{y})\},$$
o\`u $C(\D_{y})\subseteq N_{\QQ}\oplus \QQ$ est le c\^one engendr\'e par la r\'eunion de $(\D_{y}, 1)$ et de $(\sigma, 0)$. Le symbole $[y,\nu, \ell]$ d\'esigne la classe modulo la relation d'\'equivalence $\sim$ d\'efinie par $(y,\nu, \ell)\sim (y'. \nu', \ell')$  si et seulement si ($y= y'$, $\nu = \nu'$, $\ell = \ell'$) ou ($\ell = \ell' = 0$, $\nu = \nu'$). Soit $C(\D_{X,\V})$ l'ensemble $\bigcup_{y\in C_{X}}\{y\}\times C(\D_{y})$ modulo $\sim$. Sous l'hypoth\`ese $\QQ$-Gorenstein la fonction $\omega_{X}: C(\D_{X,\V})\rightarrow \QQ$ est uniquement d\'etermin\'ee par les conditions suivantes pour tout $y\in C_{X}$ (voir \cite[Section 5.2 et Proposition 5.11]{LPR16}). 
\begin{itemize}
\item[(i)]  $\omega_{X}$ induit une fonction $\QQ$-lin\'eaire sur chaque c\^one $C(\D_{y})$;
\item[(ii)] Nous avons
$\omega_{X}(y,\rho) = -1$ pour tout vecteur primitif g\'en\'erateur $\rho\in N\oplus \ZZ$ d'une face de dimension $1$ de $C(\D_{y})$ tel que
l'intersection de $\QQ_{\geq 0}\rho$ avec $ \deg(\D_{X,\V})\cup\varrho(\D_{X,\V})$ est vide;
\item[(iii)] On a $\omega_{X}(y,\alpha^{\vee}_{|M}, 0) = -a_{\alpha}$ pour tout $\alpha\in \Phi_{X}$.
\item[(iv)] Si $\rho\in N$ est un g\'en\'erateur primitif d'une face de dimension 1 de $\sigma$ tel que $\QQ_{\geq 0}\rho\cap \deg(\D_{X,\V})\neq \emptyset$, alors il existe $\lambda\in\QQ_{>0}$ tel que $\rho =\lambda v$ o\`u $v = \sum_{z\in C_{X}}v_{z}$ et $v_{z}$ est un sommet de $\D_{z}$ \cite[Lemma 5.22]{LPR16}. Dans ce cas, on a $C_{X} = C$ et 
$$\omega_{X}(y,\rho , 0) = \lambda\left(\deg\, K_{C} + \sum_{z\in C}\left(1 - \frac{1}{\kappa(v_{z})}\right)\right)\text{ avec } \kappa(v_{z}) = \inf\{\ell\in\ZZ_{>0}\,\,|\,\, \ell v_{z}\in N\}.$$
\end{itemize}
Les \'el\'ements de $|\D_{X,\V}|$ correspondent \`a certaines valuations (g\'eom\'etriques) $G$-invariantes  de $k(X)$ \`a valeurs dans $\ZZ$ dont le centre est une sous-vari\'et\'e ferm\'ee irr\'eductible $G$-stable de $X$. Pour chaque $\xi\in |\D_{X,\V}|$ correspondant \`a un diviseur exceptionnel $G$-stable d'une d\'esingularisation \'equivariante de $X$, la valeur $-1-\omega_{X}(\xi)$ est \'egale \`a sa  discr\'epance (voir \cite[Propositions 5.9 et 5.15]{LPR16}).  

\section*{R\'esultat principal}
\begin{theorem}\label{theo} Soit $X$ une $G$-vari\'et\'e normale $\QQ$-Gorenstein avec orbites horosph\'eriques. Supposons que $X$ admette un quotient rationnel $X\dashrightarrow C$, o\`u $C$ est une courbe alg\'ebrique compl\`ete lisse, et que $X$ soit d\'ecrite par le diviseur poly\'edral colori\'e $(\D_{X, \V}, \F_{X})$. Alors $X$ a au plus des singularit\'es log canoniques si et seulement si au moins une des assertions  $(a),(b),(c)$  est vraie:
\begin{itemize}
\item[$(a)$] $C_{X}$ est une courbe alg\'ebrique affine.
\item[$(b)$] $C_{X}$ s'identifie \`a la droite projective $\PP_{k}^{1}$ et $\sum_{y\in C_{X}}\left(1-\frac{1}{\kappa_{y}}\right)\leq 2,$ o\`u $\kappa_{y} = \max\{\kappa(v)\,| v \text{ sommet de }\D_{y}\}$ et $\kappa(v) = \inf\{\ell\in\ZZ_{>0}\,\,|\,\, \ell v\in N\}.$ 
\item[$(c)$] $C_{X}$ est une courbe elliptique et pour tout $y\in C_{X}$ les sommets du 
poly\`edre $\D_{y}$ sont des vecteurs entiers.
\end{itemize}
Par ailleurs, $X$ a au plus des singularit\'es canoniques (resp. terminales) si et seulement si:
\begin{itemize}
\item[$(d)$] Pour tout $y\in C_{X}$ et tout vecteur entier primitif $\xi$ tel que $\xi$ soit dans l'interieur relatif de $C_{y}(\D)$ ou un g\'en\'erateur d'une face de dimension $1$ de $C_{y}(\D)$ tel que $\QQ_{\geq 0}\xi\cap (\deg(\D_{X,\V})\cup\varrho(\D_{X,\V}))\neq \emptyset$, on a l'in\'egalit\'e $\omega_{X}(y,\xi)\leq -1$ $($resp. $\omega_{X}(y,\xi)< -1$ $)$.
\end{itemize}
\end{theorem}
\begin{proof}
Nous commen\c cons par rappeler la construction de \cite[Section 5.3]{LPR16} pour une d\'esingularisation  de $X$. Elle est donn\'ee par une suite de morphismes propres birationnels $G$-equivariants
\begin{equation*}\label{e:desingularization}
	\begin{tikzpicture}[description/.style={fill=white,inner sep=2pt},baseline=(current  bounding  box.center)]
			\matrix (m) [matrix of math nodes, row sep=3em,column sep=2.5em, text height=1.5ex, text depth=0.25ex]
		{ 	X'= X(\N') & X(\tilde \N) & X(\N) & X. \\
		};
		\path[-]
			(m-1-1) edge[->] node[auto] {$q'$} (m-1-2)
			(m-1-2) edge[->] node[auto] {$q$} (m-1-3)
			(m-1-3) edge[->] node[auto] {$\pi$} (m-1-4)
			(m-1-1) edge[->,bend right=10,looseness=1.4] node[below] {$\psi$} (m-1-4);
	\end{tikzpicture}
\end{equation*}
Les $G$-vari\'et\'es  $X(\N'), X(\tilde{\N}), X(\N)$ sont d\'ecrites par des ensembles finis de diviseurs poly\'edraux colori\'es, appel\'es \emph{\'eventails divisoriels colori\'es} (voir \cite{Tim97},\cite[Section 3]{LPR16} pour plus de d\'etails); chacun d'autres eux d\'ecrivant un recollement en des ouverts $G$-stables de $G$-vari\'et\'es normales avec orbites horosph\'eriques. On a $\N = \{(\D_{X, \V}, \emptyset)\}$ et $\pi$
est le morphisme naturel de d\'ecolorisation (voir \cite[Section 2.2]{LT16}). De plus, en consid\'erant un recouvrement fini $(C_{X}^{i})_{i\in I}$ d'ouverts affines de $C_{X}$, l'ensemble $\tilde{\N}$ est $\{({\D_{X, \V}}_{|C_{X}^{i}}, \emptyset)\,|\, i\in I\}$ et $q$ est le morphisme de \emph{contraction} ou \emph{d'affinisation} (voir par exemple \cite[Theorem 3.1]{AH06}, \cite[Section 4.2]{LPR16}). Enfin la vari\'et\'e lisse 
$X(\N')$ est induite par une subdivision r\'eguli\`ere de $C(\D_{X, \V})$ (voir \cite[Section 5.3]{LPR16} pour une construction inspir\'ee de \cite{KKMS73}).

En utilisant \cite[Proposition 5.9]{LPR16}, un diviseur canonique relatif \`a $\psi$ et port\'e sur le lieu exceptionnel peut \^etre exprim\'e comme la somme
$$K_{X'/X} = K_{X'} -\psi^{\star}K_{X} = \sum_{i=1}^{r} (-1-\omega_{X}(y_{i}, p_{i})) D_{(y_{i}, p_{i})}+ \sum_{j=1}^{s} (-1-\omega_{X}(y_{j}', p_{j}')) D_{(y_{j}', p_{j}')} +\sum_{\ell=1}^{t} (-1-\omega_{X}(y_{\ell}'', p_{\ell}'')) D_{(y_{\ell}'', p_{\ell}'')},$$
o\`u  $D_{(y_{i}, p_{i})}, D_{(y_{j}', p_{j}')}, D_{(y_{\ell}'', p_{\ell}'')}$ sont les diviseurs premiers correspondants. Ceux de la forme $D_{(y_{\ell}'', p_{\ell}'')}$ sont obtenus comme composantes irr\'eductibles du lieu exceptionnel de $q'$ et ceux de la forme $D_{(y_{i}, p_{i})}, D_{(y_{j}', p_{j}')}$ sont obtenus comme images inverses de composantes irr\'eductibles des lieux exceptionnels de $q$ et $\pi$, respectivement. Supposons que $X$ satisfasse au moins une des assertions $(a), (b), (c)$. Nous rappelons que la condition log canonique est \'equivalente  \`a ce que tous les coefficients de $K_{X'/X}$ sont $\geq -1$. Si $X$ v\'erifie $(a)$, alors $\omega_{X}\leq 0$ et donc $X$ a des singularit\'es log canoniques. Donc on peut supposer que $C_{X} = C$ est compl\`ete. Dans ce cas, pour toute suite $\vv= (v_{y})_{y\in C}$ o\`u $v_{y}$ est un sommet de $\D_{y}$ on pose  
$$\delta(\vv):=  \deg\, K_{C} + \sum_{z\in C}\left(1 - \frac{1}{\kappa(v_{z})}\right).$$
Pour tout $1\leq i\leq r$, l'in\'egalit\'e $-1-\omega_{X}(y_{i}, p_{i})\geq -1$ est \'equivalente \`a $\delta(\vv^{i})\leq 0,$ o\`u $\vv^{i} = (v_{y}^{i})_{y\in C}$ est une suite telle que $v_{y}^{i}$ est un sommet de $\D_{y}$ pour tout $y\in C$ et telle qu'il existe $\lambda\in \QQ_{>0}$ tel que $\lambda\cdot p_{i} =\sum_{y\in C} v_{y}^{i}$. Cette derni\`ere condition est consequence de la validit\'e de l'assertion $(b)$ ou $(c)$.

R\'eciproquement, supposons que $X$ a des singularit\'es log canoniques et que $C_{X} = C$ est une courbe compl\`ete lisse. Soit $\vv = (v_{y})_{y\in C}$ une suite arbitraire telle que pour tout $y\in C$, le vecteur $v_{y}$ est un sommet de $\D_{y}$. Notons $v = \sum_{y\in C}v_{y}$. Alors
en reprenant les notations ci-dessus les in\'egalit\'es $\delta(\vv^{i})\leq 0$ pour $1\leq i\leq r$ impliquent $\deg\, K_{C}\leq 0$, c'est \`a dire $\deg\, K_{C} = -2$ ou $\deg\, K_{C} = 0$. Elles impliquent en outre que $(\omega_{X})_{|\sigma}\leq 0$.
Puisque $C_{X}$ est compl\`ete, on a $(K_{X})_{|X_{0}} = \div(f\chi^{e})$ (cf. \cite[Section 2.3]{LT16}) pour le diviseur canonique $K_{X}$ d\'efini dans \cite[Theorem 2.18]{LT16}, o\`u  $f\chi^{e}$ est une fonction rationnelle $B$-propre de $k(X)$ avec $f\in k(C)^{\star}$ et $e\in M$. En adaptant l'argument de la preuve de \cite[Lemma 5.23]{LPR16}, on d\'eduit la formule $$0 \geq \omega_{X}(z, v) = \sum_{y\in C}\frac{1}{\kappa(v_{y})}\kappa(v_{y})(\langle e , v_{y}\rangle + \ord_{y}(f)) = \delta(\vv)$$ o\`u $z\in C$. 
Consid\'erons le cas $\deg\, K_{C} = 0$. Alors $C_{X}$ est une courbe elliptique et l'in\'egalit\'e $\delta(\vv)\leq 0$ donne $v_{y}\in N$ pour tout $y\in C$. Comme la suite $(v_{y})_{y\in C}$ est choisie de fa\c con arbitraire, on obtient l'assertion $(c)$. Si maintenant $\deg\, K_{C} = -2$, alors $C_{X}$ s'identifie \`a la droite projective et \`a nouveau en maximisant $\delta(\vv)\leq 0$ sur toute suite $\vv$, on a $(b)$.

Nous passons au crit\`ere pour les singularit\'es canoniques (resp. terminales). Supposons que $X$ a des singularit\'es canoniques (resp. terminales). Soit $\xi\in C_{y}(\D)$ un vecteur primitif entier o\`u $y\in C_{X}$. Alors on peut construire l'application $q'$ de sorte qu'il existe $1\leq \ell\leq t$ tel que $\xi = (y_{\ell}'', p_{\ell}'')$. En particulier, le fait que la discr\'epance au diviseur $D_{(y_{\ell}'', p_{\ell}'')}$ est $\geq 0$ (resp. $>0$) implique la condition $(d)$. La r\'eciproque est ais\'ee \`a d\'emontrer et laiss\'ee au lecteur.
\end{proof}
Dans l'exemple suivant, nous regardons le cas particulier des surfaces affines normales 
avec une op\'eration alg\'ebrique fid\`ele du groupe multiplicatif (voir \cite[Theorem 5.7]{LS13}
pour le cas des singularit\'es canoniques\footnote{Le lecteur notera que la s\'erie $E$ dans \emph{loc. cit.} devrait \^etre 
$ E_{i}: \frac{1}{2}\cdot [0] + \frac{1}{3}\cdot [1]  - \frac{i - 4}{i-3}\cdot [\infty]$
pour $i = 6,7,8$.}). Plus g\'en\'eralement, nous renvoyons \`a \cite{Sem80}, \cite[Section 7]{Ish14} pour plus de d\'etails sur les singularit\'es de surfaces. Notons que toute surface normale  ayant des singularit\'es rationnelles est $\mathbb{Q}$-Gorenstein  et celles ayant des singularit\'es terminales sont lisses (voir \cite[Theorems 7.1.18, 7.3.2]{Ish14}). 
\begin{exemple} { \em $k^{\star}$-surfaces.}
Supposons que $X$ est une surface affine normale et que le groupe op\'erant $G$ est  le groupe $\G_{m}(k)= (k^{\star},\times)$. Alors on a $X= X_{0}$,  $M = \ZZ$ et $\F_{X} =\emptyset$. 
Si $\sigma \neq 0$, alors on peut supposer que $\sigma = \QQ_{\geq 0}$ et donc $\D$ est uniquement d\'etermin\'e par le diviseur $D :=\D(1)$. Dans ce cas, on dit que la $k^{\star}$-surface $X$ est de type \emph{parabolique} si $C_{X}$ est affine et de type \emph{elliptique} si $C_{X}$ est compl\`ete. Le cas restant est lorsque $\sigma = \{0\}$, $C_{X}$ est affine, et
$\D$ est alors uniquement d\'etermin\'e par le couple $(D_{-}, D_{+})$, o\`u $D_{-} =  \D(-1)$ et $D_{+} = \D(1)$; la surface correspondante est dite de type \emph{hyperbolique}. Voir \cite{FZ03} pour une description g\'eom\'etrique de cette
trichotomie. Si $X$ est elliptique, alors d'apr\`es le th\'eor\`eme \ref{theo}, $X$ est log canonique si et seulement si l'un des deux cas suivant est vrai:
\begin{itemize}
\item[(1)] ($X$ est rationnelle) La courbe $C_{X}$ s'identifie \`a la droite projective $\PP^{1}_{k}$ et apr\`es r\'eductions par des isomorphismes $k^{\star}$-equivariants (i.e., en changeant $D$ par $D+E$ o\`u $E$ est un diviseur entier de degr\'e $0$ sur $C_{X}$), on peut \'ecrire le diviseur $D$ comme la somme
$$D = \frac{e_{1}}{m_{1}}\cdot [0]+\frac{e_{2}}{m_{2}}\cdot [1]+\frac{e_{3}}{m_{3}}\cdot [2]+\frac{e_{4}}{m_{4}}\cdot [\infty],$$
o\`u pour tout $1 \leq i\leq 4$, les entiers $e_{i},m_{i}$ sont premiers entre eux et le quadruplet $(m_{1}, m_{2}, m_{3}, m_{4})$ est de la forme $(2,2,r,1)$ (pour $r\in \ZZ_{>1}$), $(1, p, q, 1)$ (pour des entiers $p\geq q \geq 1$), $(2,3,3,1)$, $(2,3,4,1)$, $(2,3,5,1)$, $(2,3,6,1)$, $(2,4,4,1)$, $(3,3,3,1)$, et $(2,2,2,2)$.  
\item[(2)] ($X$ n'est pas rationnelle) Alors $C_{X}$ est une courbe elliptique et $D$ est un diviseur entier de degr\'e $>0$. En particulier, tout c\^one affine normale au dessus d'une courbe elliptique a des singularit\'es log canoniques; ces derniers correspondant au cas o\`u $D$ est tr\`es ample, c'est \`a dire satisfaisant la condition $\deg \,D\geq 3$.   
\end{itemize}  
\end{exemple}

\end{document}